\documentclass[12pt,notitlepage,a4paper]{article}
\usepackage{amsmath,amssymb,amsthm,amscd}
\usepackage[arrow,matrix]{xy}
\theoremstyle{plain}
\numberwithin{equation}{section}
\newtheorem{thm}{Theorem}[section]
\newtheorem{cor}[thm]{Corollary}
\newtheorem{lem}[thm]{Lemma}
\newtheorem{prop}[thm]{Proposition}
\newtheorem{conj}{Conjecture}
\theoremstyle{definition}
\newtheorem{df}[thm]{Definition}

\newcommand{\mb}{\mathbb}
\newcommand{\mf}{\mathfrak}
\newcommand{\ml}{\mathcal}
\newcommand{\en}{{\rm End}}
\newcommand{\ch}{{\rm CH}}
\begin{document}
\author{Rin Sugiyama\footnote{The author is supported by JSPS Research Fellowships for Young Scientists.}\\ \textit{\small Graduate School of Mathematics, Nagoya University,}\\ \textit{\small Furo-cho, Chikusa-ku, Nagoya 464-8602, Japan}\\ {\small e-mail: rin-sugiyama@math.nagoya-u.ac.jp}
}
\title{Lefschetz classes of simple factors of Fermat Jacobian of prime degree over finite fields}
\date{}
\maketitle
\begin{abstract}
We give a necessary and sufficient condition in terms of a matrix for which all Tate classes are Lefschetz for simple abelian varieties over an algebraic closure of a finite field. As an application, we prove under an assumption that all Tate classes are Lefschetz for simple factors of Fermat Jacobian of prime degree.
\end{abstract}

\bigskip


\section{Introduction}
Let $p$ be a  prime number. Let $\mb{F}_p$ be a finite field with $p$-elements and let $\mb{F}$ be a fixed algebraic closure of $\mb{F}_p$. Let $\ell$ be a prime number different from $p$. Let $A_0$ be an abelian variety over a finite subfield $\mb{F}_q$ of $\mb{F}$ and let $A$ be the abelian variety $A_0\otimes_{\mb{F}_q}\mb{F}$ over $\mb{F}$. There is the cycle map 
$$cl_A^i: \ch^i(A)\otimes\mb{Q}_{\ell}\longrightarrow H^{2i}(A,\mb{Q}_{\ell}(i)),$$
where $\ch^i(A)$ is the Chow group of algebraic cycles on $A$ of codimension $i$ modulo rational equivalence, and $H^{2i}(A,\mb{Q}_{\ell}(i))$ is the $\ell$-adic \'etale cohomology of $A$. Then we know that the image of the map $cl_A^i$ is contained in the space $\ml{T}_{\ell}^i(A)$ of \textit{$\ell$-adic Tate classes} of degree $i$ on $A$ which is defined as follows:
$$\ml{T}_{\ell}^i(A):=\varinjlim_{L/\mb{F}_q}H^{2i}(A,\mb{Q}_{\ell}(i))^{{\rm Gal}(\mb{F}/L)}.$$
Here $L/\mb{F}_q$ runs over all finite extensions of $\mb{F}_q$. We call the elements of the image of $cl_A^i$ \textit{the algebraic classes} of degree $i$.
\begin{conj} For all $i\ge0$, 
${\rm Im}(cl_A^i)=\ml{T}_{\ell}^i(A).$
\end{conj}
This is conjectured by Tate \cite[Conjecture 1]{Ta}. In this paper, we call the conjecture the Tate conjecture. The Tate conjecture for $A$ implies the Tate conjecture for $A_0/\mb{F}_q$, that is, the cycle map $cl_{A_0}^i:\ch^i(A_0)\otimes\mb{Q}_{\ell}\longrightarrow H^{2i}(A,\mb{Q}_{\ell}(i))^{{\rm Gal}(\mb{F}/\mb{F}_q)}$ is surjective. On the other hand, Tate \cite{Ta1} proved that $cl_{A_0}^1$ is surjective for all abelian varieties over finite fields. Therefore all Tate classes of degree 1 are divisor classes on $A$. The elements of the $\mb{Q}_{\ell}$-subalgebra of $\displaystyle \ml{T}_{\ell}(A):=\bigoplus_{i\ge0}\ml{T}_{\ell}^i(A)$ generated by $\ml{T}_{\ell}^1(A)$ are called the \textit{the Lefschetz classes} on $A$. If all $\ell$-adic Tate classes on $A$ are Lefschetz, then the Tate conjecture holds for $A$. But there are examples where not all Tate classes are Lefschetz and the Tate conjecture holds (\cite[Example 1.8]{Mi3}). If $A$ is a product of elliptic curves, then all $\ell$-adic Tate classes on $A$ are Lefschetz by a result of Spiess \cite{Sp}. Zarhin \cite{Za}, Lenstra and Zarhin \cite{LZ} gave other example (c.f. \cite[A.7]{Mi3}). Kowalski \cite{Ko} proved that for certain simple ordinary abelian varieties, all $\ell$-adic Tate classes are Lefschetz.  In this paper, for a simple factor of the Jacobian of Fermat curve of prime degree, we give a necessary and sufficient condition for which all $\ell$-adic Tate classes are Lefschetz.  
 
Let $m$ be a positive integer prime to $p$. Let $J_m$ be the Jacobian of Fermat curve of degree $m$ defined by the equation
$$x_0^m+x_1^m+x_2^m=0 \subset \mb{P}^2_{\mb{F}_q}.$$
Shioda-Katsura \cite[Proposition 3.10]{SK} proved that $J_m$ is isogenous to a product of supersingular elliptic curves if and only if $p^{\nu}\equiv -1 \mod m$ for some $\nu$. If $m$ is a prime $l$ different from $p$, then the condition that $p^{\nu}\equiv -1 \mod l$ for some $\nu$ is equivalent to the condition that the residual degree $f$ of $p$ in $\mb{Q}(\mu_l)$ is even. Hence if $f$ is even, then a simple factor $A$ of $J_l$ is isogenous to a supersingular elliptic curve over $\mb{F}$. In this case, all $\ell$-adic Tate classes on all powers of $A$ are Lefschetz. In this paper, we consider in case that $f$ is odd.

To state main result, we recall Jacobi sum. Let $\alpha$ be an element of the set
$$\ml{A}_{l}^1:=\bigl\{ (a_0,a_1,a_2) \ | \ a_i \in \mb{Z}/l, \ a_i \not \equiv 0, \ a_0+a_1+a_2\equiv 0\bigr\}.$$
Then Jacobi sum $j(\alpha)$ is defined by 

$$j(\alpha)=-\sum_{\stackrel{1+v_1++v_2=0}{ \ v_i \in \mb{F}_{p^f}^{\times}}}\psi(v_1)^{a_1}\psi(v_2)^{a_2}.$$
Here $\psi$ is a fixed character of order $l$ of $\mb{F}_{p^f}^{\times}$. By definition, $j(\alpha) \in \mb{Q}(\mu_l)$ where $\mu_l$ is the set of $l$-th root of unity. We identify $(\mb{Z}/l)^{\times}$ with the Galois group of $\mb{Q}(\mu_l)/\mb{Q}$.

Let $A_0$ be an absolute simple factor of the Jacobian of $V_l$ over $\mb{F}_{p^f}$ and let $A:=A_0\otimes\mb{F}$. We denote by $C(A)$ the center of $\en_{\mb{F}}(A)\otimes \mb{Q}$. Let $\pi_0$ be the Frobenius endomorphism of $A_0$. Then $\pi_0$ can be given by Jacobi sum (cf. \cite{We}): $\pi_0=j(\alpha)$ for some $\alpha \in \ml{A}_{l}^1$. Therefore $C(A)=\mb{Q}(j(\alpha))$ and $C(A)\subset \mb{Q}(\mu_l)$. 

Our main result is the following:
\begin{thm}\label{t}
Let $l$ be an odd prime number different from $p$. Let $A$ be a simple factor of the Jacobian of the Fermat curve of degree $l$ over $\mb{F}$. Let $\alpha=(a_0,a_1,a_2)$ be an element of $\ml{A}_l^1$ such that $C(A)=\mb{Q}(j(\alpha))$. Let $H_{\alpha}$ be the Galois group of $\mb{Q}(\mu_l)/C(A)$. Then all $\ell$-adic Tate classes on all powers of $A$ are Lefschetz if and only if for any odd Dirichlet character $\chi$ of $(\mb{Z}/l)^{\times}$ with $\chi|_{H_{\alpha}}=1$, $\displaystyle \sum_{i=0}^2\chi(a_i)\neq0$.

In particular, all $\ell$-adic Tate classes on all powers of $A$ are Lefschetz in the following cases:

$(1)$ $[C(A):\mb{Q})] \not \equiv 0 \mod 6$,

$(2)$ $\alpha$ is equal to an element of the form $(a,a,b)$ up to permutation,

$(3)$ $[C(A):\mb{Q})]=2^{s+1}\cdot3$ $(s\geq0)$.
\end{thm}
\begin{cor}[Corollary \ref{cc2}]\label{cc1}
Let $l$ be an odd prime number different from $p$. Let $J(C_l)$ be the Jacobian variety of the hyperelliptic curve 
$$C_l\;:\;y^2=x^l-1.$$
Then all $\ell$-adic Tate classes on all powers of $J(C_l)$ are Lefschetz.
\end{cor}
This corollary is an analogue of Shioda's result on Hodge conjecture for $J(C_l)/\mb{C}$ (\cite[Corollary 5.3]{Sh}). In case that $p\equiv1\mod l$,  Shioda's argument in \cite{Sh} works over finite fields and shows that the above corollary holds. The argument also shows that not all $\ell$-adic Tate classes are Lefschetz and the Tate conjecture holds for $J(C_9)$. However in case that $p\not \equiv 1 \mod l$, the argument needs a similar result to the key lemma in Shioda's argument which is not proven (cf. \cite[p.\ 181, Lemma]{Sh2}). In this paper,  we use other argument. More precisely we use Milne's result \cite{Mi3} on the Tate conjecture (see \S2). 

The key of the proof of Theorem \ref{t} is to give a necessary and sufficient condition for which all $\ell$-adic Tate classes are Lefschetz for simple abelian varieties in terms of a matrix as follows: for a simple abelian variety $A$ over $\mb{F}$, we define a matrix $T_A$ whose rank depends only on $A$ (see \S 3). Using Milne's result, we prove the following:

\begin{thm}[Theorem \ref{t2}]\label{t2i}
Let $A$ be a simple abelian variety over $\mb{F}$ of dimension $\geq2$. Then all $\ell$-adic Tate classes on all powers of $A$ are Lefschetz  if and only if ${\rm rank}\ T_A=r$. Here $r$ is the half of the number of all distinct embeddings $C(A)\rightarrow \mb{C}$.
\end{thm}

This paper is organized as follows: in \S2, we recall Milne's result on the Tate conjecture. In \S3, we prove Theorem \ref{t2i} and we calculate the matrix $T_A$ in a special case (Theorem \ref{t3}). In the last section, we prove Theorem \ref{t} using the necessary and sufficient condition.

\subsection*{Notation}

Through this paper, $\mb{Q}^{\rm al}$ denotes the algebraic closure of $\mb{Q}$ in $\mb{C}$. For a finite \'etale $\mb{Q}$-algebra $E$, $\Sigma_E:={\rm Hom}(E,\mb{Q}^{\rm al})$. If $E$ is a field Galois over $\mb{Q}$, we identify $\Sigma_E$ with the Galois group ${\rm Gal}(E/\mb{Q})$.

For a finite set $S$, $\mb{Z}^S$ denotes the set of functions $f:S\rightarrow \mb{Z}$. 

An affine algebraic group is of multiplicative type if it is  commutative and its identity component is a torus. For such a group $W$ over $\mb{Q}$, $\chi(W):={\rm Hom}(W_{\mb{Q}^{\rm al}},\mb{G}_m)$ denotes the group of characters of $W$.

For a finite \'etale $\mb{Q}$-algebra $E$, $(\mb{G}_m)_{E/\mb{Q}}$ denotes the Weil restriction ${\rm Res}_{E/\mb{Q}}(\mb{G}_m)$ which is characterized by $\chi((\mb{G}_m)_{E/\mb{Q}})=\mb{Z}^{\Sigma_E}$.

For an abelian variety $A$ over $\mb{F}$, $\en^0_{\mb{F}}(A)$ denotes $\en_{\mb{F}}(A)\otimes \mb{Q}$, and $C(A)$ denotes the center of $\en^0_{\mb{F}}(A)$.

\section{Milne's result on the Tate conjecture}
We recall Milne's result on the Tate conjecture. For an abelian variety $A$ over $\mb{F}$, there are three important groups of multiplicative type $L(A)$, $M(A)$ and $P(A)$ which act on $\displaystyle H^{2*}(A):=\bigoplus_{i\ge0}H^{2i}(A,\mb{Q}_{\ell}(i))$. These groups are introduced by Milne \cite{Mi1}, \cite{Mi2}, \cite{Mi3}, who proved that these groups are characterized by the following properties (see \cite[p.\ 14, Lemma]{Mi3}): for all $r$,
\begin{itemize}
\item[(a)] $H^{2*}(A^r)^{L(A)}$ is the space of Lefschetz classes.

\item[(b)] $H^{2*}(A^r)^{M(A)}$ is the space of algebraic classes, provided numerical equivalence coincides with $\ell$-adic homological equivalence.

\item[(c)] $H^{2*}(A^r)^{P(A)}$ is the space of $\ell$-adic Tate classes on $A^r$.
\end{itemize}
Statement (a) is proved by Milne \cite{Mi1}. Statement (b) is proved by using a result of Jannsen \cite{Ja}. Clozel \cite{Co} proved that for an abelian variety $A$ over $\mb{F}$, there is a set of primes $\ell$ of positive density for which $\ell$-adic homological equivalence and numerical equivalence coincide. Deligne \cite{D} gives a more precise version of Clozel's result. Statement (c) is almost by definition of $P(A)$. The following theorem is due to Milne \cite[p.\ 14, Theorem]{Mi3}:
\begin{thm}\label{t1}
Let $A$ be an abelian variety over $\mb{F}$. Then $P(A)\subset M(A)\subset L(A)$, and

{\rm (i)} the Tate conjecture holds for $A$ if and only if $P(A)=M(A)${\rm ;}

{\rm (ii)} all $\ell$-adic Tate classes on all powers of $A$ are Lefschetz if and only if $P(A)=L(A)$. 
\end{thm}

We now recall the definitions of the Lefschetz group and the group $P$ associated to an abelian variety over $\mb{F}$, and recall a description of the character group of these groups. For properties and results on the groups, we refer to Milne \cite{Mi1}, \cite{Mi2} and \cite{Mi3}. 

Let $A$ be an abelian variety over $\mb{F}$. A polarization $\lambda: A\rightarrow A^{\vee}$ of $A$ determines an involution  of $\en^0_{\mb{F}}(A)$ which stabilizes $C(A)$. The restriction of the involution to $C(A)$ is independent of the choice of $\lambda$. By $\dagger$, we denote this restriction on $C(A)$. 

\begin{df}[\text{\cite[4.3, 4.4]{Mi1}, \cite[pp.52-53]{Mi2}, \cite[A.3]{Mi3}}]
The \textit{Lefschetz group} $L(A)$ of $A$ is the algebraic group over $\mb{Q}$ such that
\begin{align*}
L(A)(R)=\{\alpha \in (C(A)\otimes R)^{\times} \ | \ \alpha \alpha^{\dagger} \in R^{\times}\}
\end{align*}
 for all $\mb{Q}$-algebras $R$.
\end{df}
We can describe $L(A)$ as a subgroup of $(\mb{G}_m)_{C(A)/\mb{Q}}$ in terms of characters as follows (\cite[A.7]{Mi3}): $L(A)$ is a subgroup of $(\mb{G}_m)_{C(A)/\mb{Q}}$ whose character group is 
\begin{align}\label{cl}
\frac{\mb{Z}^{\Sigma_{C(A)}}}{\{g \in \mb{Z}^{\Sigma_{C(A)}}\ \ | \ g=\iota g \ \text{and} \ \sum g(\sigma)=0\}}.
\end{align}
Here $\iota g$ is a function sending an element $\sigma$ of $\Sigma_{C(A)}$ to $g(\iota \sigma)$, and $\sum g(\sigma)$ denotes $\displaystyle \sum _{\sigma \in \Sigma_{C(A)}}g(\sigma)$.

\begin{df}[\text{\cite[\S 4]{Mi2}, \cite[A.7]{Mi3}}]
Let $A_0$ be a model of $A$ over a finite field $\mb{F}_q \subset \mb{F}$, and let $\pi_0$ be the Frobenius of $A_0$. Then the group $P(A)$ of $A$ is the smallest algebraic subgroup of $L(A)$ containing some power of $\pi_0$. It is independent of the choice of $A_0$.
\end{df}
To state Milne's result on the character group of $P$, we introduce Weil numbers and some notion which is related to Weil numbers. A Weil $q$-number of weight $i$ is an algebraic number $\pi$ such that $q^N\pi$ is an algebraic integer for some $N$ and the complex absolute value $|\sigma(\pi)|$ is $q^{i/2}$, for all embeddings $\sigma :\mb{Q}[\pi]\rightarrow\mb{C}$. Then $\pi$ is a unit at all primes of $\mb{Q}[\pi]$ not dividing $p$. We define the {\it slope function} $s_{\pi}$ of $\pi$ as follows: for any prime $\mf{p}$ dividing $p$ of a field containing $\pi$,
\begin{align}\label{sf}
s_{\pi}(\mf{p})=\frac{{\rm ord}_{\mf{p}}(\pi)}{{\rm ord}_{\mf{p}}(q)}.
\end{align}
For the definition of Weil numbers, $s_{\pi}(\mf{p})+s_{\pi}(\iota \mf{p})=i (=wt(\pi))$. Here $\iota$ is complex conjugation on $\mb{C}$. The slope function determines a Weil $q$-number up to a root of unity. 

Let $\pi$ be a Weil $p^f$-number and let $\pi^{\prime}$ be a Weil $p^{f^{\prime}}$-number. We say $\pi$ and $\pi^{\prime}$ are {\it equivalent} if  $\pi^{f^{\prime}}={\pi^{\prime}}^f\cdot \zeta$ for some root of unity $\zeta$. We define a {\it Weil germ} to be an equivalent class of Weil numbers. For a Weil germ $\pi$, the slope function and weight of $\pi$ are the slope function (see \eqref{sf}) and weight of any representative of $\pi$, and $\mb{Q}[\pi]$ is defined to be the smallest subfield of $\mb{Q}^{\rm al}$ containing a representative of $\pi$. 

Now assume that $A$ is simple and that $\en^0_{\mb{F}}(A)=\en^0_{\mb{F}_q}(A_0)$. Then the Frobenius endomorphism $\pi_0$ of $A_0$ generates $C(A)$ over $\mb{Q}$, i.e. $C(A)=\mb{Q}[\pi_0]$. We know that the Frobenius endomorphism $\pi_0$ of $A_0$ is a Weil $q$-number of weight 1. Let $\pi_A$ denote the germ represented by $\pi_0$. Milne's result on the character of $P(A)$ is the following (\cite[A.7]{Mi3}): let $g$ be a character of $L(A)$. Then $g$ is trivial on $P(A)$ if and only if for all primes $v$ dividing $p$ of a field containing all conjugates $\sigma (\pi_0)$,
\begin{align}\label{cp}
\sum_{\ \ \sigma \in \Sigma_{C(A)}} g(\sigma) s_{\sigma \pi_A}(v)=0.
\end{align}
Note that $s_{\sigma \pi_A}(v)=s_{\pi_A}(\sigma ^{-1}v)$.

\section{Necessary and sufficient condition}
Let $A$ be a simple abelian variety over $\mb{F}$. Let $A_0$ be a model of $A$ over a finite subfield $\mb{F}_q\subset \mb{F}$ with property that $\en_{\mb{F}_q}(A_0)=\en_{\mb{F}}(A)$. We take a finite Galois extension $K$ of $\mb{Q}$ containing all conjugates of $C(A)$. Let $\{\mf{p}_1,\dots,\mf{p}_d\}$ be the set of primes of $K$ dividing $p$. We assume that $\dim A\geq2$. Then $C(A)$ is totally imaginary and CM field (cf. \cite[p.\,97]{TaH}). Let $\{\sigma_1,\dots,\sigma_{r},\iota\sigma_1,\dots,\iota\sigma_{r}\}$ be the set of all distinct embeddings $C(A) \rightarrow \mb{C}$. We define a $d\times r$ matrix $T_A$ as follows:
\begin{align*}
T_A=(a_{ij}),\;\;a_{ij}:=
s_{\sigma_j\pi_0}(\mf{p}_i)-\frac{1}{2} \ \ &\text{($1\le i\le d$, $1\leq j\leq r$)}.\end{align*}
Then the matrix $T_A$ is independent of the choice of the model $A_0/\mb{F}_q$, but depends on the ordering of the $\mf{p}_i$ and the $\sigma_j$.  The rank of $T_A$ depends only on $A$. Using Milne's result on the character group of $L(A)$ and $P(A)$, we prove the following: 
\begin{thm}\label{t2}
Let $A$ be a simple abelian variety over $\mb{F}$ of dimension $\geq2$. Then $P(A)=L(A)$ if and only if ${\rm rank}\ T_A=r$.
\end{thm}
We first describe the kernel of the natural map $\mb{Z}^{\Sigma_{C(A)}}\rightarrow \chi(L(A))$ in terms of a matrix. Let $J$ be the $(r+1) \times 2r$ matrix
\begin{align*}
\begin{pmatrix}
I_r&-I_r\\
0&B
\end{pmatrix},
\ \ \ B=
\begin{pmatrix}
1&\cdots&1
\end{pmatrix}
\end{align*}
where
$I_r$ is the $r\times r$ identity matrix. We consider the set of character functions for $\sigma_i, \iota\sigma_i$ as a basis of $\mb{Z}^{\Sigma_{C(A)}}$. We then have the following:
\begin{lem}\label{l1}
The kernel of the natural map $\phi : \mb{Z}^{\Sigma_{C(A)}}\rightarrow \chi(L(A))$ coincides with the kernel of the map $\mb{Z}^{\Sigma_{C(A)}}\rightarrow\mb{Z}^{r+1}$ defined by $J$.
\end{lem}
\begin{proof}
From \eqref{cl}, the kernel of $\phi$ is $\{g \in \mb{Z}^{\Sigma_{C(A)}} \ |\ g=\iota g \ \text{and} \ \sum g(\sigma)=0\}$. Here $\iota g$ is a function sending an element $\sigma$ of $\Sigma_{C(A)}$ to $g(\iota \sigma)$, and $\sum g(\sigma)$ denotes $\displaystyle \sum _{\sigma \in \Sigma_{C(A)}}g(\sigma)$. Therefore the kernel of $\phi$ is the kernel of the map $\mb{Z}^{\Sigma_{C(A)}}\rightarrow\mb{Z}^{2r+1}$ defined by the $(2r+1) \times 2r$ matrix
\begin{align*}
J^{\prime}:=
\begin{pmatrix}
I_r&-I_r\\
-I_r&I_r\\
B&B
\end{pmatrix}.
\end{align*}
By row operations, $J^{\prime}$ is equivalent over $\mb{Z}$ to  
\begin{align*}
\begin{pmatrix}
I_r&-I_r\\
0&2B\\
0_r&0_r
\end{pmatrix}
\end{align*}
where $0_r$ is the $r\times r$ matrix with all entries equal to $0$. From this we have ${\rm Ker}(J)={\rm Ker}(J^{\prime})$.
\end{proof}

Next we describe the condition \eqref{cp} in terms of a matrix. Let $T^{\prime}$ be the $d \times 2r$ matrix $(U \ V)$, where $U$ is the $d\times r$ matrix $(s_{\sigma_j \pi_0}(\mf{p}_i))$ and $V$ is the $d\times r$ matrix $(s_{\iota \sigma_j \pi_0}(\mf{p}_i))$ ($1\leq i\leq d, 1\leq j\leq r$). From \eqref{cp}, a character $g$ of $L(A)$ is trivial on $P(A)$ if and only if $g$ belongs to the kernel of the map $\mb{Z}^{\Sigma_{C(A)}}\rightarrow\mb{Z}^d$ defined by $T^{\prime}$. By Lemma \ref{l1}, $L(A)=P(A)$ if and only if ${\rm Ker}(J)={\rm Ker}(T^{\prime})$. Hence we deduce Theorem \ref{t2} from the following lemma:
\begin{lem}
Let notation be as above. \\
$(1)$ ${\rm rank} \ T^{\prime}= {\rm rank} \ T_A+1$. \\
$(2)$ ${\rm Ker}(J)={\rm Ker}(T^{\prime})$ if and only if ${\rm rank} \ T^{\prime}= r+1$
\end{lem}
\begin{proof}
(1) From the equation $s_{\sigma_j \pi_0}(\mf{p}_i)+s_{\sigma_j\pi_0}(\iota \mf{p}_i)=1$ for all $i,j$, we easily see that the matrix $T^{\prime}$ is column equivalent to the $d\times 2r$ matrix 
\begin{align*}
\begin{pmatrix}
C&0_{d\times (r-1)}
\end{pmatrix},
\ \ \ C=
\begin{pmatrix}
&1\\
U&\vdots\\
&1
\end{pmatrix}
\end{align*}
where $0_{d\times (r-1)}$ is the $d \times (r-1)$ matrix with all entries equal to $0$. 

Since the complex conjugation $\iota \in G$ acts on the set $\{\mf{p}_1,\dots,\mf{p}_d\}$, after renumbering the $\mf{p}_i$ if necessary,  there is a positive integer $t$ such that 
\begin{align*}
\iota \mf{p}_i&=\mf{p}_{i+t} \ \ \ \ \text{for $1\leq i\leq t$},\\
\iota \mf{p}_i&=\mf{p}_i \ \ \ \ \ \ \ \text{for $2t+1\leq i \leq d$}.
\end{align*}
From this, we obtain that for each $j$,
\begin{equation}\label{s}
\begin{aligned}
s_{\sigma_j \pi_0}(\mf{p}_i)+s_{\sigma_j\pi_0}(\mf{p}_{i+t})&=1\ \ \ \ \ \ \  \text{for $1\leq i\leq t$},\\
s_{\sigma_j \pi_0}(\mf{p}_i)&=\frac{1}{2} \ \ \ \ \ \ \text{for $2t+1\leq i \leq d$}.
\end{aligned}
\end{equation}
These equations show that the matrix $C$ is row equivalent to the matrix
\begin{align*}
\begin{pmatrix}
U^{\prime}&0_{(t+1)\times (r-1)}\\
0_{(d-t+1)\times (r+1)}&0_{(d-t+1)\times (r-1)}
\end{pmatrix},
\ \ \ U^{\prime}=
\begin{pmatrix}
&&&0\\
&U^{''}&&\vdots\\
&&&0\\
1&\cdots&1&2
\end{pmatrix}
\end{align*}
where $U^{''}$ is the $t\times r$ matrix $(s_{\sigma_j \pi_0}(\mf{p}_i)-\frac{1}{2})$ ($1\leq i\leq t, 1\leq j\leq r$).

On the other hand, from \eqref{s} we also obtain that the matrix $T_A$ is row equivalent to the matrix
\begin{align*}
\begin{pmatrix}
&U^{''}&\\
&0_{(d-t)\times r}
\end{pmatrix}.
\end{align*}
Hence
\begin{align*}
{\rm rank} \ T^{\prime}={\rm rank} \ C ={\rm rank} \ U^{''} +1={\rm rank} \ T_A+1.
\end{align*}

(2) If ${\rm Ker}(J)={\rm Ker}(T^{\prime})$, then clearly ${\rm rank} \ T^{\prime} =r+1$.

Conversely assume that ${\rm rank} \ T^{\prime}=r+1$. Then $T^{\prime}$ is row equivalent to a matrix of the following form
\begin{align*}
\begin{pmatrix}
I_{r+1}&D\\
0_{(d-r-1)\times (r+1)}&0_{(d-r-1)\times (r-1)}
\end{pmatrix}.
\end{align*}
Put $D=(b_{ij})$. From the equation $s_{\sigma_j \pi_0}(\mf{p}_i)+s_{\iota\sigma_j\pi_0}(\mf{p}_i)=1$, we have
\begin{align*} 
1+b_{ii}&=1 \ \ \text{if $i\neq1,r+1$},\\
0+b_{ij}&=1 \ \ \text{otherwise}.
\end{align*}
Hence the matrix $(I_{r+1} \ D)$ is row equivalent to the matrix $J$, which implies ${\rm Ker}(T^{\prime})={\rm Ker}((I_{r+1}\ D))={\rm Ker}(J)$.
\end{proof}

\subsection{Calculation in special case}
Applying Theorem \ref{t1} and Theorem \ref{t2}, we have the following:
\begin{thm}\label{t3}
Let $A$ be a simple abelian variety over $\mb{F}$ of dimension $\geq2$. Assume that $C(A)/\mb{Q}$ is abelian with Galois group $G$. Then Then all $\ell$-adic Tate classes on all powers of $A$ are Lefschetz if and only if for any character $\varphi$ of $G$ with $\varphi(\iota)=-1$,
\begin{align*}
\sum_{\sigma \in G}e(\sigma)\varphi(\sigma)\neq0.
\end{align*}
Here $e(\sigma)=s_{\pi_0}(\sigma\mf{p})-\frac{1}{2}$ where $\mf{p}$ is a prime of $C(A)$ dividing $p$.

In particular, all $\ell$-adic Tate classes on all powers of $A$ are Lefschetz if one of the following condition holds{\rm :}

{\rm (1)} $G$ is cyclic of order $2^{s+1}$ $(s\geq0)$,

{\rm (2)} $G$ is cyclic of order $2^{s+1}l$ with $l$ odd prime and the order of $\en^0_{\mb{F}}(A)$ in the Brauer group ${\rm Br}(C(A))$ of $C(A)$ is odd.
\end{thm}
To prove this theorem, we need the following proposition:
\begin{prop}\label{p1}
Let $A$ be a simple abelian variety over $\mb{F}$. Assume that $C(A)$ is Galois over $\mb{Q}$. Let $G$ be the Galois group of $C(A)$ over $\mb{Q}$. Let $\mf{p}$ be a prime ideal of $C(A)$ lying over $p$. If the decomposition group $G_{\mf{p}}$ of $\mf{p}$ in $C(A)$ is a normal subgroup of $G$, then $G_{\mf{p}}=1$, namely $p$ is completely decomposed in $C(A)$.
\end{prop}
\begin{proof}
Let a prime decomposition of $\pi_0$ in $C(A)$ be as follows:
$$(\pi_0)=\prod_{\sigma \in G/G_{\mf{p}}}\sigma(\mf{p})^{e_{\sigma}}.$$
Let $\tau$ be an element of $G_{\mf{p}}$. Since $G_{\mf{p}}$ is normal, we have $(\pi_0)=(\tau\pi_0)$ as ideals. From Lemma \ref{lw} below, we obtain that $\tau =1$ and that $G_{\mf{p}}=1$.
\end{proof}
\begin{lem}\label{lw}
Let $\pi$ be a Weil germ. Let $\pi_0\in \mb{Q}[\pi]$ be a representative of $\pi$. Let $K$ be a Galois extension of $\mb{Q}$ such that $\mb{Q}[\pi]\subset K$. Then there is no elements $\sigma \in {\rm Gal}(K/\mb{Q})$ satisfying the following conditions{\rm :}

\textup{(1)} $\sigma$ fix the ideal $(\pi_0)$,

\textup{(2)} $\sigma$ is not trivial on $\mb{Q}[\pi]$. 
\end{lem}
\begin{proof}
We assume that there is an element $\sigma \in {\rm Gal}(K/\mb{Q})$ satisfying conditions (1) and (2). From condition (1), there is an unit $u$ of the integer ring of $K$ such that $\pi_0=u\cdot\sigma\pi_0$. For any $\tau \in {\rm Gal}(K/\mb{Q})$, we have $|\tau u|=1$ since $|\tau \pi_0|=q^{1/2}$. Here we used that $\pi_0$ is a Weil $q$-number. Hence $u$ is a root of unity. Therefore we have $\pi_0^m=\sigma \pi_0^m$ for some $m>0$. Since $\sigma $ acts on the subfield $\mb{Q}(\pi_0^m)$ of $\mb{Q}[\pi]$ trivially, $\mb{Q}(\pi_0^m)$ is not equal to $\mb{Q}[\pi]$ by condition (2). 

On the other hand, since $\mb{Q}[\pi]$ is the smallest field containing a representative of $\pi$, we obtain that $\mb{Q}[\pi]\subset \mb{Q}(\pi_0^m)$. Hence $\mb{Q}(\pi_0^m)=\mb{Q}[\pi]$ which is a contradiction.
\end{proof}
\renewcommand{\proofname}{\textit{Proof of Theorem \ref{t3}}}
\begin{proof}
By Theorem \ref{t1} and Theorem \ref{t2}, our task is reduced to calculate the matrix $T_A$. Since $C(A)$ is Galois over $\mb{Q}$, we may take $K=C(A)$. Put $G=\{\sigma_1,\dots,\sigma_{r},\iota\sigma_1,\dots,\iota\sigma_{r}\}$. Here $\iota \in G$ is the complex conjugate on $C(A)$. Let $\mf{p}$ be a prime of $C(A)$ dividing $p$. By Proposition \ref{p1}, the set $\{\sigma\mf{p} \;|\; \sigma \in G\}$ is the set of all primes of $C(A)$ dividing $p$. Let $e$ be the function $G\longrightarrow \mb{C}$ defined as $e(\sigma)=s_{\pi_0}(\sigma\mf{p})-\frac{1}{2}$ for $\sigma \in G$. From the proof of Theorem \ref{t2}, the matrix $T_A$ is row equivalent to the matrix
\begin{align*}
\begin{pmatrix}
&U^{''}&\\
&0_{r\times r}
\end{pmatrix}
\end{align*}
where $U^{''}$ is the $r\times r$ matrix $\big(e(\sigma_i\sigma_j^{-1})\big)$ ($1\leq i, j\leq r$). 
Since ${\rm rank} \ T_A={\rm rank} \ U^{''}$, we obtain that ${\rm rank} \ T_A=r$ if and only if $\det (U^{''})\neq 0$. 

Now we calculate $\det (U^{''})$. Let $\phi$ be the fixed character of $G$ with $\phi(\iota)=-1$. Let $f:G\longrightarrow \mb{C}$ be the function defined as $f(\sigma)=\phi(\sigma)e(\sigma)$ for $\sigma \in G$. Then we have
\begin{align*}
\det (U^{''})=\det \Big(\big(f(\sigma_i\sigma_j^{-1})\big)\Big).
\end{align*}
For any $\sigma \in G$, we have
$$f(\iota \sigma)=\phi(\iota\sigma)e(\iota\sigma)=-\phi(\sigma)(-e(\sigma))=f(\sigma).$$
Hence $f$ is a function on $G^{'}:=G/\{1, \iota\}$.
Now we need the following lemma:
\begin{lem}
Let $G$ be a finite abelian group and let $f$ be a function on $G$ with values in some field of characteristic $0$. Then
\begin{align*}
\det \big(f(\sigma \tau^{-1})\big)_{\sigma, \tau \in G}=\prod_{\psi} \sum_{\sigma \in G}f(\sigma)\psi(\sigma).
\end{align*}
Here $\psi$ runs over all character of $G$.
\end{lem}
For the proof of this lemma, see \cite[Lemma 5.26]{Wa}.
 
From this lemma, we have
\begin{align}\label{d}
\det \big(f(\sigma \tau^{-1})\big)_{\sigma, \tau \in G^{'}}=\prod_{\psi} \sum_{\sigma \in G^{'}}f(\sigma)\psi(\sigma).
\end{align}
Here $\psi$ runs over all character of $G^{'}$. Furthermore, by elementary calculation, we have
$$\text{RHS of \eqref{d}}=\frac{1}{2}\prod_{\varphi} \sum_{\sigma \in G}e(\sigma)\varphi(\sigma).$$
Here $\varphi$ runs over all character of $G$ with $\varphi(\iota)=-1$. By the above argument, ${\rm rank} \ T_A=r$ if and only if for any such $\varphi$,
\begin{align*}
\sum_{\sigma \in G}e(\sigma)\varphi(\sigma)\neq0.
\end{align*}
This completes the proof of the first assertion of Theorem \ref{t3}.
 
Now we put $\displaystyle E(\varphi):=\sum_{\sigma \in G}e(\sigma)\varphi(\sigma)$ and prove that $E(\varphi)\neq0$ in some cases. Let $g$ be a generator of $G$. In case that condition (1) holds, then $\displaystyle E(\varphi)=2\sum_{i=0}^{2^s-1}e(g^i)\varphi(g^i)$. Since $\varphi(g)$ is a primitive $2^{s+1}$-th root of unity, we have $[\mb{Q}(\zeta):\mb{Q}]=2^s$. Therefore the set $\{1,\zeta,\zeta^2,\dots,\zeta^{2^s-1}\}$ is a base of $\mb{Q}(\zeta)$ over $\mb{Q}$. Therefore if $E(\varphi)=0$, then we have
\begin{align*}
e(1)=e(g)=\cdots=e(g^{2^s-1})=0.
\end{align*}
From this, we have 
$(\pi_0)=(p^n)$ as ideal in $C(A)$ which is fixed by the action of $\iota \in G$. Here $n:={\rm ord}_{\mf{p}}(\pi_0)$. By Lemma \ref{lw}, we have $\iota=1$, which is a contradiction. Therefore $E(\varphi)\neq0$.

Next we consider in case that the order of $G$ is $2^{s+1}l$ with $l$ an odd prime. Then $\varphi(g)$ is a primitive $2^{s+1}$-th root of unity or a primitive $2^{s+1}l$-th root of unity. If $\varphi(g)$ is a primitive $2^{s+1}l$-th root of unity and if $E(\varphi)=0$, then we have
\begin{align*}
\sum_{i=0}^{r2^sl-1}e(g^i)x^i=h(x)\cdot (b_1x^{2^s-1}+b_2x^{2^s-2}+\cdots+b_{2^s-1}x+b_{2^s})
\end{align*}
for some $b_i\in \mb{Q}$. Here $h(x):=x^{2^s(l-1)}-x^{2^s(l-2)}+\cdots-x^{2^s}+1$ is the minimal polynomial to $\varphi(g)$. From this equation, we obtain that for $0\leq t\leq 2^s-1$ and for $0\leq i\leq l-1$,
\begin{align}\label{a}
e(g^{2^si+t})=(-1)^ie(g^{t}).
\end{align}
Then the ideal $(\pi_0)$ is fixed by $g^l$ because equation \eqref{a} holds. This is a contradiction by Lemma \ref{lw}. Hence $E(\varphi)\neq0$.

If $\varphi(g)$ is a 
primitive $2^{s+1}$-th root of unity and if $E(\varphi)=0$, then by a similar argument in case that $\varphi(g)$ is a primitive $2^{s+1}l$-th root of unity, for each $0\leq t\leq 2^s-1$ we have
\begin{align*}
\sum_{i=0}^{l-1} (-1)^ie(g^{2^si+t})=0.
\end{align*}

On the other hand, let $n$ be the order of $\en^0_{\mb{F}}(A)$ in the Brauer group ${\rm Br}(C(A))$ of $C(A)$. By class field theory, $n$ is the smallest integer such that $n\cdot {\rm inv}_{v}(\en^0_{\mb{F}}(A))$ belongs to $\mb{Z}$ for all prime $v$ of $C(A)$. We here have
\begin{align*}
{\rm inv}_{v}(\en^0_{\mb{F}}(A))=s_{\pi}(v)[C(A)_v:\mb{Q}_p].
\end{align*}
By Proposition \ref{p1}, $[C(A)_v:\mb{Q}_p]=1$. Hence, for any $1\leq t\leq 2^s$ we have
\begin{align*}
n\cdot\sum_{i=0}^{l-1} (-1)^ie(g^{2^si+t})\equiv \frac{n}{2} \mod \mb{Z}.
\end{align*}
Since $n$ is odd from the assumption, $\displaystyle \sum_{i=0}^{l-1} (-1)^ie(g^{2^si+t})\neq0$. Therefore $E(\varphi)\neq0$.
\end{proof}
\renewcommand{\proofname}{\textit{Proof}}

\section{Proof of Theorem \ref{t}}
We here prove Theorem \ref{t}. We introduce some notation: For any $c \in (\mb{Z}/l)^{\times}$, we denote by $\langle c\rangle$ the least natural number such that $\langle c\rangle\equiv c \mod l$. We write $H$ for the subgroup of $(\mb{Z}/l)^{\times}$ generated by $p$. We identify $(\mb{Z}/l)^{\times}$ with the Galois group of $\mb{Q}(\mu_l)/\mb{Q}$.
\renewcommand{\proofname}{\textit{Proof of Theorem \ref{t}}}
\begin{proof}
By a result of Shioda-Katsura mentioned in Introduction, we may assume that the residual degree $f$ of $p$ in $\mb{Q}(\mu_l)$ is odd and that $\dim A\ge2$. By Gonz\'alez's result \cite[Theorem 3.3]{G}, we have
\begin{align}
H_{\alpha}=\Big\{ c \in (\mb{Z}/l)^{\times} \;\Big|\; \sum_{h \in H}\sum_{i=0}^2\langle hca_i\rangle=\sum_{h \in H}\sum_{i=0}^2\langle ha_i\rangle\Big\}.
\end{align}
Let $G$ be the Galois group of $C(A)/\mb{Q}$. Then we have
$$G\simeq(\mb{Z}/l)^{\times}/H_{\alpha}.$$
Let $\iota \in G$ be the complex conjugation on $C(A)$. Then to give a character $\varphi$ of $G$ with $\varphi(\iota)=-1$ is equivalent to give a odd Dirichlet character $\chi$ of $(\mb{Z}/l)^{\times}$ with $\chi|_{H_{\alpha}}=1$. By Theorem \ref{t3}, it suffices to show that for any odd Dirichlet character $\chi$ of $(\mb{Z}/l)^{\times}$ with $\chi|_{H_{\alpha}}=1$, $\displaystyle \sum_{\sigma \in G}e(\sigma)\chi(\sigma)=0$ if and only if $\displaystyle \sum_{i=0}^2\chi(a_i)=0$. Here $e(\sigma)=s_{j(\alpha)}(\sigma \mf{p})-\frac{1}{2}$ where $\mf{p}$ is a fixed prime of $C(A)$ dividing $p$.

Therefore for a such Dirichlet character $\chi$, we calculate $\displaystyle \sum_{\sigma \in G}e(\sigma)\chi(\sigma)$. We first consider $e(\sigma)$. Let $\mf{q}$ be a prime of $\mb{Q}(\mu_l)$ dividing $\mf{p}$. Then we have 
$$e(\sigma)=s_{j(\alpha)}(c\mf{q})-\frac{1}{2}.$$ 
Here $c \in (\mb{Z}/l)^{\times}$ is a representative of $\sigma$. Hence we have
\begin{align}\label{fe}
\sum_{\sigma \in G}e(\sigma)\chi(\sigma)=\frac{1}{d}\sum_{c \in (\mb{Z}/l)^{\times}}e(c)\chi(c),
\end{align}
where $d$ is the cardinality of $H_{\alpha}$ and $e(c)=s_{j(\alpha)}(c\mf{q})-\frac{1}{2}$. By the ideal decomposition of $j(\alpha)$ in $\mb{Q}(\mu_l)$ (cf. \cite{SK}), we have
\begin{align*}
s_{j(\alpha)}(c\mf{q})-\frac{1}{2}&=\frac{1}{f}\sum_{h \in H}\sum_{i=0}^2\Big(\frac{\langle ha_ic^{-1}\rangle}{l}-1\Big)-\frac{1}{2}\\
&=\frac{1}{fl}\sum_{i=0}^2\sum_{h \in H}\Big(\langle ha_ic^{-1}\rangle-\frac{l}{2}\Big).
\end{align*}
From this equation and \eqref{fe},
\begin{align*}
\frac{1}{d}\sum_{c \in (\mb{Z}/l)^{\times}}e(c)\chi(c)&=\frac{1}{fld}\sum_{c \in (\mb{Z}/l)^{\times}}\sum_{i=0}^2\sum_{h \in H}\Big(\langle ha_ic^{-1}\rangle-\frac{l}{2}\Big)\chi(c)\\
&=\frac{1}{fld}\sum_{i=0}^2\sum_{h \in H}\sum_{c \in (\mb{Z}/l)^{\times}}\Big(\langle ha_ic^{-1}\rangle-\frac{l}{2}\Big)\chi(c)\\
&=\frac{1}{fld}\sum_{i=0}^2\sum_{h \in H}\chi(a_i)^{-1}\sum_{c \in (\mb{Z}/l)^{\times}}\Big(\langle c\rangle-\frac{l}{2}\Big)\chi(c)^{-1}\\
&=\frac{1}{ld}\sum_{i=0}^2\chi(a_i)^{-1}\sum_{c \in (\mb{Z}/l)^{\times}}\langle c\rangle\chi(c)^{-1}.
\end{align*}
Now our task is reduced to show that $\displaystyle \sum_{c \in (\mb{Z}/l)^{\times}}\langle c\rangle\chi(c)^{-1}\ne0$. Let $L(s,\chi)$ be the $L$-series attached to $\chi$. Then $L(1,\chi)\neq0$ by \cite[Corollary 4.4]{Wa}. Furthermore by \cite[Theorem 4.9]{Wa}, we have
$$\sum_{c \in (\mb{Z}/l)^{\times}}\langle c\rangle\chi(c)^{-1}=b\cdot L(1,\chi),\;\;(b\in \mb{C}^{\times}).$$
Hence the proof of the first assertion of the theorem is complete.

Next we consider in case (1). Let $g$ be a generator of $(\mb{Z}/l)^{\times}$. Let $\chi$ be an odd Dirichlet character of $(\mb{Z}/l)^{\times}$. Then ${\rm Ker}(\chi)$ is a subgroup of $(\mb{Z}/l)^{\times}$ generated by $g^{m}$ for some $m|(l-1)$. Now $l-1=mk$ for some $k>0$. If $\displaystyle \sum_{i=0}^2\chi(a_i)=0$, then $1+\chi(a_1a_0^{-1})+\chi(a_2a_0^{-1})=0.$
Since $\chi(c)$ is a root of unity, by elementary computation $\chi(a_1a_0^{-1})$ is a primitive cubic root of unity. Hence $m$ is divided by $3$.  Moreover since $\chi(-1)=-1$ and $-1\equiv g^{\frac{l-1}{2}}$, we obtain that $m$ does not divide $\frac{l-1}{2}$. Hence $m$ is even and $k$ is odd. Therefore $m \equiv 0 \mod 6$. 

From the above argument, we see that if $l-1\not \equiv 0 \mod6$, then the order of $\chi$ is not divided by $6$. Therefore we may assume that $l-1\equiv 0\mod 6$. Let $H^{'}$ be the subgroup of $(\mb{Z}/l)^{\times}$ generated by $g^6$. If $[C(A):\mb{Q}]\not \equiv 0 \mod 6$,  then $H_{\alpha}\not \subset H^{'}$. For any odd character $\chi$ of $(\mb{Z}/l)^{\times}$ with $\chi|_{H_{\alpha}}=1$, we obtain that ${\rm Ker}(\chi)\not \subset H^{'}$. Therefore the order of $\chi$ is not divided by $6$. From the above argument, we see that $\displaystyle \sum_{i=0}^2\chi(a_i)\not=0$ for any such $\chi$. Hence the assertion follows from the first assertion of the theorem. 

Case (2) is easy. Let $\chi$ be an odd Dirichlet character of $(\mb{Z}/l)^{\times}$. If $\alpha=(a_0,a_1,a_2)=(a,a,b)$ up to permutation and if $\displaystyle \sum_{i=0}^2\chi(a_i)=0$, then
$\chi(a^{-1}b)=-2$. This is a contradiction because $\phi(c)$ is a root of unity.

Lastly we consider in case (3). From Theorem \ref{t3}\.(2), it suffices to show that the order $n$ of $\en^0_{\mb{F}}(A)$ in the Brauer group ${\rm Br}(C(A))$ of $C(A)$ is odd. By the proof of Theorem \ref{t3}, the order $n$ is the smallest integer $n$ such that $n\cdot s_{j(\alpha)}(v)$ belongs to $\mb{Z}$ for all prime $v$ of $C(A)$. Now $s_{j(\alpha)}(v)=\frac{m}{f}$ for some $m \in \mb{Z}$. Therefore $n$ is a divisor of $f$. Since $f$ is odd,  $n$ is also odd.
\end{proof}

\begin{cor}[Corollary \ref{cc1}]\label{cc2}
Let $l$ be an odd prime number different from $p$. Let $J(C_l)$ be the Jacobian variety of the hyperelliptic curve 
$$C_l\;:\;y^2=x^l-1.$$
Then all $\ell$-adic Tate classes on all powers of $J(C_l)$ are Lefschetz.
\end{cor}
\renewcommand{\proofname}{\textit{Proof}}
\begin{proof}
The hyperelliptic curve $C_l$ is a quotient of the Fermat curve $V_l$ (cf. \cite[\S5]{Sh}). By Gonz\'alez's result \cite[Theorem 3.3]{G}, $J(C_l)$ is a power of a simple abelian variety $A$ such that $C(A)=\mb{Q}(j(\alpha))$ for some $\alpha=(a,a,b)$ (cf. \cite{Sh}). Now the assertion follows from Theorem \ref{t}\.(2).
\end{proof}
\noindent\textbf{Acknowledgements} \ The author expresses his gratitude to Professors Thomas Geisser and Kanetomo Sato for many helpful suggestions and comments, and to Daisuke Shiomi for valuable discussion. Thanks are also due to the referee. In particular, Corollary \ref{cc1} was suggested by him/her.

\end{document}